\documentclass[12pt]{article}

\usepackage[english]{babel}
\usepackage{graphicx}
\usepackage{latexsym}
\usepackage{amsfonts}
\usepackage{dsfont}
\usepackage{amscd}
\usepackage{amssymb}
\usepackage{amsmath}
\usepackage{hyperref}
\usepackage{mathrsfs}
\usepackage{amsthm}
\usepackage{tikz}

\input xy
\xyoption{all}

\newcommand{\NN}{\mathbb{N}}

\newcommand{\s}{\mathcal{S}}
\newcommand{\p}{\varphi}
\newcommand{\g}{\mathcal{G}}
\newcommand{\Ef}{\widehat E_0}
\newcommand{\Et}{\widehat E_{\text{tight}}}
\newcommand{\gt}{\mathcal{G}_{\text{tight}}}

\theoremstyle{definition}
\newtheorem{theo}{Theorem}[section]
\newtheorem{rmk}[theo]{Remark}

\newtheorem{lem}[theo]{Lemma}
\newtheorem{defn}[theo]{Definition}

\newtheorem{prop}[theo]{Proposition}
\date{}
\begin{document}
\title{Amenable actions of inverse semigroups}
\author{Ruy Exel\thanks{Partially supported by CNPq (Brazil).} \hspace{1cm} Charles Starling\thanks{Supported by CNPq (Brazil).}}
\maketitle

\begin{abstract}
We say that an action of a countable discrete inverse semigroup on a locally compact Hausdorff space is amenable if its groupoid of germs is amenable in the sense of Anantharaman-Delaroche and Renault. We then show that for a given inverse semigroup $\s$, the action of $\s$ on its spectrum is amenable if and only if every action of $\s$ is amenable. 
\end{abstract}

\section{Introduction}
There are numerous ways to generalize the notion of a group. One such generalization is that of an {\em inverse semigroup}, that is, a semigroup $\s$ for which each $s\in\s$ has a unique ``inverse'' $s^*$, in the sense that $ss^*s = s$ and $s^*ss^* = s^*$. In this note, we address the question of what it means for an inverse semigroup to be amenable.

Amenability of discrete groups is an active and lively area of research. There are many equivalent definitions for what it means for a group to be amenable, and so those who attempt to define amenability for inverse semigroups have had many potential definitions to choose from. As discussed in \cite{Pa99} and \cite{Mi10}, some of the more familiar notions of amenability for groups, such as the existence of a left translation-invariant mean, produce unsatisfactory answers when applied to inverse semigroups. The definition of group amenability that motivates this work is given by the following equivalent statements for a discrete group $G$:
\begin{itemize}\addtolength{\itemsep}{-0.5\baselineskip}
\item[(a)] $G$ is amenable;
\item[(b)] the action of $G$ on a point is amenable;
\item[(c)] every action of $G$ on a locally compact Hausdorff space is amenable;
\end{itemize}
(see \cite[Example 2.7]{AD02} for the definitions and details).

In generalizing the above, we first define what it means for an action of an inverse semigroup to be amenable,
Definition \ref{amenableactiondef}. Essentially, an action is amenable if its {\em groupoid of germs} \eqref{germs} is
amenable in the sense of \cite{AR00}. We then show (Theorem \ref{maintheorem}) that (b) and (c) above are equivalent for
an inverse semigroup when ``a point'' in (b) is replaced with ``its spectrum''. This change is natural for two
reasons. The first is that if a group is viewed as an inverse semigroup, then its spectrum is a point. The second is
that the action of an inverse semigroup on a point may not be well-defined. In this way, we believe that amenability of the action of a inverse semigroup on its spectrum is a good candidate for the definition of amenability.

In \cite{Mi10}, it is argued that {\em weak containment} is a natural notion of amenability for an inverse semigroup. The conditions of Theorem \ref{maintheorem} imply weak containment (see Remark \ref{weakcontainment}), however a recent example of Willett \cite{Wi15} suggests that the converse situation might be more delicate.

This short note is organized as follows. Before proving our main result, in Section 2 we define the notion of a {\em $d$-bijective groupoid homomorphism} and show that the existence of such a map from a groupoid $\g$ into an amenable groupoid implies that $\g$ is also amenable. This is used in proving our main result in Section 3; to prove Theorem \ref{maintheorem} we construct a $d$-bijective map between the two relevant groupoids.

\section{Amenable groupoids and $d$-bijective homomorphisms}
In this section we prove a preliminary result about amenability of \'etale groupoids. After giving the necessary background, we define a certain type of groupoid homomorphism, such that if we have such a map from a groupoid $\g$ to an amenable groupoid, then $\g$ must also be amenable.

Recall that a {\em groupoid} is a set $\g$ with a distinguished subset $\g^{(2)} \subset \g \times \g$, called the set of {\em composable pairs}, a product map $\g^{(2)} \to \g$ with $(\gamma, \eta)\mapsto \gamma\eta$, and an inverse map from $\g$ to $\g$ with $\gamma \mapsto \gamma^{-1}$ such that
\begin{enumerate}\addtolength{\itemsep}{-0.5\baselineskip}
\item $(\gamma^{-1})^{-1} = \gamma$ for all $\gamma\in \g$,
\item If $(\gamma, \eta), (\eta, \nu)\in \g^{(2)}$, then $(\gamma\eta,\nu), (\gamma, \eta\nu)\in \g^{(2)}$ and $(\gamma\eta)\nu = \gamma(\eta\nu)$,
\item $(\gamma, \gamma^{-1}), (\gamma^{-1},\gamma)\in \g^{(2)}$, and $\gamma^{-1}\gamma\eta = \eta$, $\xi\gamma\gamma^{-1} = \xi$ for all $\eta, \xi$ with $(\gamma, \eta), (\xi,\gamma) \in \g^{(2)}$.
\end{enumerate}
The set of {\em units} of $\g$ is the subset $\g^{(0)}$ of elements of the form $\gamma\gamma^{-1}$. The maps $r: \g\to \g^{(0)}$ and $d:\g\to \g^{(0)}$ given by
\[
r(\gamma) = \gamma\gamma^{-1}, \hspace{1cm} d(\gamma) = \gamma^{-1}\gamma
\]
are called the {\em range} and {\em source} maps respectively. One sees that $(\gamma, \eta)\in \g^{(2)}$ is equivalent to
$r(\eta) = d(\gamma)$. One thinks of a groupoid $\g$ as a set of ``arrows'' between elements of $\g^{(0)}$. An arrow $\gamma$ ``starts'' at $d(\gamma)$ and ``ends'' at $r(\gamma)$.

A map $\varphi: \g\to \mathcal{H}$ between groupoids is called a {\em groupoid homomorphism} if $(\gamma, \eta)\in \g^{(2)}$ implies that $(\varphi(\gamma), \varphi(\eta))\in \mathcal{H}^{(2)}$ and $\varphi(\gamma\eta) = \varphi(\gamma)\varphi(\eta)$. A short calculation shows that this implies that $\varphi(\gamma^{-1}) = \varphi(\gamma)^{-1}$, and so $\varphi(\g^{(0)}) \subset \mathcal{H}^{(0)}$, $r\circ\varphi = \varphi\circ r$, and $d\circ\varphi = \varphi \circ d$.

A {\em topological groupoid} is a groupoid which is a topological space where the inverse and product maps are
continuous, where we are considering $\g^{(2)}$ with the product topology inherited from $\g\times\g$. A topological
groupoid is called {\em \'etale} if it is locally compact, its unit space is Hausdorff, and the range and source maps are local homeomorphisms. These properties imply that $\g^{(0)}$ is open. Furthermore, in a second countable \'etale groupoid, the spaces
\[
\g_x : = d^{-1}(x), \hspace{1cm} \g^x := r^{-1}(x)
\]
are discrete for all $x\in \g^{(0)}$. We note that an \'etale groupoid may not be Hausdorff, even though we always assume the unit space is.

\def\comment#1{(\color{red}#1\color{black})}

The following theorem from \cite[Corollary 3.3.8]{AR00} will be used  as our  definition of amenability for a second countable \'etale groupoid.

\begin{theo}\label{amenablegroupoid}
Let $\mathcal{G}$ be a second countable \'etale groupoid. The following are equivalent:
\begin{enumerate}\addtolength{\itemsep}{-0.5\baselineskip}
\item $\mathcal{G}$ is amenable.
\item There exists a sequence $(g_n)$ of Borel functions on $\g$ %in $B_b(\mathcal{G})$
such that
\begin{enumerate}\addtolength{\itemsep}{-0.5\baselineskip}
\item the function $x\mapsto \sum_{r(\gamma) = x}|g_n(\gamma)|$ is bounded;
\item for all $x\in \mathcal{G}^{(0)}$ and $n\in\NN$ we have $\sum_{r(\gamma) = x}g_n(\gamma) = 1$; and
\item for all $\gamma\in\mathcal{G}$ the sequence
\[
\sum_{r(\gamma) = r(\eta)}\left|g_n(\gamma^{-1}\eta) - g_n(\eta)\right|
\]
converges to 0 with $n$.
\end{enumerate}
\item There exists a sequence $(f_n)$ of Borel functions on $\g$ %in $B_b(\mathcal{G})$
such that
\begin{enumerate}\addtolength{\itemsep}{-0.5\baselineskip}
\item the function $x\mapsto \sum_{d(\gamma) = x}|f_n(\gamma)|$ is bounded;
\item  for all $x\in \mathcal{G}^{(0)}$ and $n\in\NN$ we have $\sum_{d(\gamma) = x}f_n(\gamma) = 1$; and
\item for all $\gamma\in\mathcal{G}$ the sequence
\[
\sum_{d(\gamma) = d(\eta)}\left|f_n(\eta\gamma^{-1}) - f_n(\eta)\right|
\]
converges to 0 with $n$.
\end{enumerate}
\end{enumerate}
\end{theo}
\begin{proof}
(1)$\Leftrightarrow$(2) is \cite[Corollary 3.3.8]{AR00}, and (2)$\Leftrightarrow$(3) follows by composing the given functions with the groupoid inverse map and redefining variables.
\end{proof}

We now define a type of groupoid homomorphism which arises naturally when considering inverse semigroup actions.

\begin{defn}\label{dbijectivedef} Let $\g$ and $\mathcal{H}$ be groupoids. We say that a groupoid homomorphism $\varphi:
\g\to \mathcal{H}$ is {\em $d$-bijective} (or {\em source-bijective}) if for all $x\in \g^{(0)}$, the restriction
$\varphi: \g_x \to \mathcal{H}_{\varphi(x)}$ is bijective. We will say that $\varphi$ is {\em $r$-bijective} (or {\em
range-bijective}) if for all $x\in \g^{(0)}$, the restriction $\varphi: \g^x \to \mathcal{H}^{\varphi(x)}$ is bijective.
\end{defn}

The definition of a groupoid is symmetric with respect to the range and source map, so the following should not be surprising.

\begin{lem}
A map $\varphi:\g\to \mathcal{H}$ is $d$-bijective if and only if it is $r$-bijective. 
\end{lem}
\begin{proof}
Take $x\in \g^{(0)}$. From the definitions, one sees that $(\g_x)^{-1} = \{\gamma^{-1}\mid \gamma\in \g_x\} = \g^x$. Because groupoid homomorphisms commute with the inverse map, it is also clear that $\left.\varphi\right|_{\g_x}\circ ^{-1} = \left.\varphi\right|_{\g^x}$. If $\varphi: \g_x \to \mathcal{H}_{\varphi(x)}$ is bijective, then so is $\left.\varphi\right|_{\g_x}\circ ^{-1} = \left.\varphi\right|_{\g^x}$. The other direction is analogous.
\end{proof}

\begin{prop}\label{dbijectiveamenable}
Suppose that $\g$ and $\mathcal{H}$ are \'etale groupoids, and that $\p: \g\to \mathcal{H}$ is a $d$-bijective Borel map. If $\mathcal{H}$ is amenable, then so is $\g$.
\end{prop}

\begin{proof}
Suppose that we have $\varphi: \g \to \mathcal{H}$ as in the statement, and that $(f_n)$ is a sequence of functions on $\mathcal{H}$ as in Theorem \ref{amenablegroupoid}.3. For each $n\in \NN$, define $h_n$ on $\g$ by
\[
h_n := f_n\circ \varphi.
\]
Because $\varphi$ and the $f_n$ Borel, the $h_n$ are Borel as well. Take $x\in \g^{(0)}$, and calculate
\begin{eqnarray*}
\sum_{\gamma\in \g_x}h_n(\gamma) &= & \sum_{\gamma\in \g_x} f_n(\varphi(\gamma))\\
& =&\sum_{\nu\in \mathcal{H}_{\varphi(x)}} f_n(\nu)\\
&=& 1.
\end{eqnarray*}
The second line above is due to the fact that $\varphi$ is a bijection between $\g_x$ and $\mathcal{H}_{\varphi(x)}$. Hence, the $h_n$ satisfy Theorem \ref{amenablegroupoid}.3(b). A similar calculation shows that the $h_n$ satisfy Theorem \ref{amenablegroupoid}.3(a).

Finally, fix $\gamma\in \g$ and let $x = d(\gamma)$. We calculate
\begin{eqnarray*}
\sum_{\eta\in \g_x}\left|h_n(\eta\gamma^{-1}) - h_n(\eta)\right| &=&\sum_{\eta\in \g_x}\left|f_n(\varphi(\eta\gamma^{-1})) - f_n(\varphi(\eta))\right|\\
&=& \sum_{\varphi(\eta)\in \mathcal{H}_{\varphi(x)}} \left|f_n(\varphi(\eta)\varphi(\gamma^{-1})) - f_n(\varphi(\eta))\right|\\
&=& \sum_{\nu\in \mathcal{H}_{\varphi(x)}} \left|f_n(\nu\varphi(\gamma^{-1})) - f_n(\nu)\right|\\
&\stackrel{n\to \infty}{\to}& 0.
\end{eqnarray*}
Again, the third line is due to the fact that $\varphi$ is a bijection between $\g_x$ and $\mathcal{H}_{\varphi(x)}$. Thus, the $h_n$ satisfy Theorem \ref{amenablegroupoid}.3(c) and we are done.
\end{proof}
We are grateful to the referee for pointing out that the above can also be seen as a corollary of \cite[Theorem 5.3.14]{AR00}.

\section{Amenable inverse semigroup actions}
In this section we define inverse semigroups and their actions. To each action of an inverse semigroup one may associate an \'etale groupoid, and we will say that an action is amenable if the groupoid associated to it is amenable in the sense of the last section. We then use Proposition \ref{dbijectiveamenable} to prove our main result: that if a certain universal action of an inverse semigroup is amenable, then all of its actions are amenable.

By an {\em inverse semigroup} we mean a semigroup $\s$ such that for each $s\in\s$ there is a unique $s^*\in\s$ such
that $s^*ss^* = s^*$ and $ss^*s = s$. For convenience, we will also assume that there is an element $0\in\s$ such that
$0s = s0 = 0$ for all $s\in\s$. If a given inverse semigroup does not have a zero element (such as in the case of a
group), one may simply adjoin a zero element to it and extend the multiplication in the obvious way -- the resulting semigroup is an inverse semigroup with zero. An element $e\in \s$ is called an {\em idempotent} if $e^2 = e$. For each $s\in \s$, the elements $s^*s$ and $ss^*$ are idempotents. The set of idempotents forms a commutative subsemigroup of $\s$. From now on we will denote by $E$ the set of idempotent elements of an inverse semigroup without making specific reference to the inverse semigroup.  Any inverse semigroup we consider is assumed to be countable and discrete.

Let $X$ be a set. Then the {\em symmetric inverse monoid} on $X$ is the set $\mathcal{I}(X)$ of bijections between subsets of $X$:
\[
\mathcal{I}(X) = \{ f: U\to V\mid U, V\subset X, f\text{ bijective}\}.
\]
This becomes an inverse semigroup when given the operation of composition of functions on largest domain possible. If $f, g\in \mathcal{I}(X)$ are such that the range of $g$ does not intersect the domain of $f$, then the product $fg$ is the empty function, which is the zero element of $\mathcal{I}(X)$.

We now recall the definition of this note's principal object of study.
\begin{defn}
An {\em action} of an inverse semigroup $\s$ on a locally compact Hausdorff space $X$ is a semigroup homomorphism
$\alpha: \s \to \mathcal{I}(X)$ such that for each $s$, the map $\alpha_s$ is continuous and its domain is open in $X$,
and the union of all the domains of the $\theta_s$ coincides with $X$. We also require that $\alpha_0$ is the empty map on $X$. If $\alpha$ is an action of $\s$ on $X$, we will write $\alpha: \s \curvearrowright X$.
\end{defn}
We note that for an action $\alpha: \s \curvearrowright X$ and an idempotent $e\in E$, the map $\alpha_e$ is necessarily the identity map on its domain, which we denote $D_e^\alpha\subset X$. Furthermore, for $s\in \s$, the domain of the function $\alpha_s$ coincides with the domain of the idempotent $\alpha_{s^*s}$, and so we write  $\alpha_s: D^\alpha_{s^*s}\to D^\alpha_{ss^*}$. For each $s\in \s$, the inverse of $\alpha_s$ is $\alpha_{s^*}$ which is continuous by definition, and so each $\alpha_s$ is a homeomorphism.

Given an action $\alpha$ of an inverse semigroup $\s$ on a space $X$ we can form a groupoid which encodes the action. The {\em groupoid of germs} for such an action is
\begin{equation}\label{germs}
\g(\alpha) = \{ [s,x] \mid s\in\s, x\in D^\alpha_{s^*s}\}
\end{equation}
where two elements $[s,x]$ and $[t,y]$ are equal if and only if $x=y$ and there exists $e\in E$ such that $x\in D^\alpha_e$ and $se = te$. The groupoid operations are given by
\[
[s,x]^{-1} = [s^*, \alpha_s(x)], \hspace{0.5cm} r([s,x]) = \alpha_s(x), \hspace{0.5cm} d([s,x]) = x, \hspace{0.5cm} [t, \alpha_s(x)][s,x] = [ts, x].
\]
For $s\in \s$ and an open set $U\subset D_{s^*s}^\alpha$, let
\[
\Theta(s, U) = \{[s, x]\mid x\in U\}.
\]
Sets of this form generate a topology on $\g(\alpha)$, and under this topology $\g(\alpha)$ is \'etale, and because $\s$ is countable, $\g(\alpha)$ is second countable. For a more detailed discussion of inverse semigroup actions and groupoids of germs, the interested reader is directed to \cite[\S 4]{Ex08}.

Given this construction, we make the following definition.
\begin{defn}\label{amenableactiondef}
We say that an action $\alpha$ of an inverse semigroup $\s$ on a locally compact Hausdorff space $X$ is {\em amenable} if the groupoid of germs $\g(\alpha)$ is amenable.
\end{defn}

An inverse semigroup $\s$ acts on an intrinsic space built from a natural order structure on its idempotent set. For $e,
f\in E$, we write $e\leqslant f$ if $ef = e$; this defines a partial order on $E$.
A {\em filter} in $E$ is a nonempty subset $\xi\subset E$ which does not contain the zero element, is closed under the product, and such that if $e\in \xi$ and $f\in E$ with $e\leqslant f$, then $f\in \xi$.
  The set of all filters will be denoted $\Ef$, and can be viewed as a subset of the product space $\{0,1\}^E$. We give
$\Ef$ the relative topology from this space. If $X, Y\subset E$ are finite subsets of $E$, define \[ U(X, Y) = \{ \xi\in
\Ef\mid X\subset \xi, Y\cap \xi = \varnothing\}.  \] The collection of such sets forms a basis for the topology on
$\Ef$. The space $\Ef$ is called the {\em spectrum} of $\s$.

We now define the intrinsic action $\theta$ of $\s$ on $\Ef$. For $e\in E$, let
\[
D_e^\theta:= U(\{e\}, \varnothing) = \{\xi\in \Ef\mid e\in \xi\}
\]
and define $\theta_s: D^\theta_{s^*s}\to D^\theta_{ss^*}$ by
\[
\theta_s(\xi) = \{e\in E\mid sfs^*\leqslant e\text{ for some }f\in \xi\}.
\]
The groupoid of germs for this action $\g(\theta)$ is sometimes called the {\em universal groupoid} for $\s$. This groupoid was defined in \cite{Pa99}.

We can now state our main theorem.
\begin{theo}\label{maintheorem}
Let $\s$ be an inverse semigroup. Then the following are equivalent.
\begin{enumerate}\addtolength{\itemsep}{-0.5\baselineskip}
\item[(a)] The canonical action $\theta: \s\curvearrowright \Ef$ is amenable.
\item[(b)] Every action of $\s$ on a locally compact Hausdorff space is amenable.
\end{enumerate}
\end{theo}

Of course, the difficult part of the above is to prove that the groupoid of germs of a given action $\alpha$ is amenable assuming that $\g(\theta)$ is. To do this, for any action $\alpha:\s\curvearrowright X$ we produce a $d$-bijective map from $\g(\alpha)$ to $\g(\theta)$ and appeal to Proposition \ref{dbijectiveamenable}. 

{\bf For the duration of this paper, we fix an inverse semigroup $\s$ and an action $\alpha:\s\curvearrowright X$ of $\s$.}

For each $x\in X$, the relation $s\sim_x t$ if and only if $[s,x] = [t,x]$ is an equivalence relation on the set of all $s$ in $\s$ such that $x\in D^\alpha_{s^*s}$.   The equivalence class of an element $s$ will be denoted $[s]_x^\alpha$ and the
set of all equivalence classes will be denoted $[\s]^\alpha_x$. The set $[\s]^\alpha_x$ can be thought of as a partition of the set $\{s\in\s\mid x\in D^\alpha_{s^*s}\}$.

We define a map $\rho: X\to \Ef$ by
\begin{equation}\label{rhodef}
\rho(x) = \{ e\in E\mid x\in D^\alpha_e\}.
\end{equation}
It is immediate that for each $x\in X$, the set $\rho(x)$ is a filter in $E$. We also have the following facts about $\rho$:

\begin{lem}\label{rhofacts}
Let $\rho:X \to \Ef$ be as in \eqref{rhodef}. Then:
\begin{enumerate}\addtolength{\itemsep}{-0.5\baselineskip}
\item For all $x\in X$ and $e\in E$, $x\in D_e^\alpha$ if and only if $\rho(x)\in D_e^\theta$.
\item For all $x\in X$, we have $\{s\in\s\mid x\in D_{s^*s}^\alpha\} = \{s\in\s\mid \rho(x)\in D_{s^*s}^\theta\}$.
\item For all $x\in X$ and $s, t$ in the set referred to above,  we have that $[s]_x^\alpha = [t]_x^\alpha$ if and only if $[s]_{\rho(x)}^\theta = [t]_{\rho(x)}^\theta$.
\item For all $x\in X$ and $s\in \s$ with $x\in D_{s^*s}^\alpha$, we have that $\rho(\alpha_s(x)) = \theta_s(\rho(x))$.
\end{enumerate}
\end{lem}
\begin{proof}
\begin{enumerate}\addtolength{\itemsep}{-0.5\baselineskip}
\item For $x\in X$, $x\in D_e^\alpha$ if and only if $e\in \rho(x)$, which is equivalent to saying that $\rho(x)\in D_e^\theta$.
\item This is a direct consequence of 1.
\item Take $x\in X$ and $s, t\in \s$ such that $x\in D^\alpha_{s^*s}\cap D^\alpha_{t^*t}$. Then $[s, x] = [t, x]$ if and only if there exists $e\in E$ such that $x\in D^\alpha_e$ and $se = te$, which by 1 is equivalent to $x\in D^\theta_e$ and $se = te$.
\item If $e\in \rho(\alpha_s(x))$, then $\alpha_s(x)\in D_e^\alpha$, which implies that $x\in D_{s^*es}^\alpha$, and so $\rho(x)\in D_{s^*es}^\theta$. Hence, $\theta_s(\rho(x)) \in D_{ss^*ess^*}^\theta\subset D_{e}^\theta$, so we have that $e\in \theta_s(\rho(x))$.

Conversely, if $e\in \theta_s(\rho(x))$, then there is an idempotent $f$ such that $x\in D_f^\alpha$ and $sfs^* \leqslant e$, which is to say $sfs^*e = sfs^*$. This implies that $D_{sfs^*}^\alpha\subset D_e^\alpha$, and so $\alpha_s(x)\in D_{e}^\alpha$, whence $e\in \rho(\alpha_s(x))$.
\end{enumerate}
\end{proof}

The map $\rho$ induces a map $\tilde\rho: \g(\alpha) \to \g(\theta)$ defined by $\tilde\rho([s,x]) = [s, \rho(x)]$. This map is well-defined by Lemma \ref{rhofacts}.3.
\begin{lem}\label{rhohomo}
The map $\tilde\rho: \g(\alpha) \to \g(\theta)$ is a $d$-bijective groupoid homomorphism.
\end{lem}
\begin{proof}
We first check that $\tilde\rho$ is a groupoid homomorphism. We only have $([s,x], [t,y])\in \g(\alpha)^{(2)}$ if $y = \alpha_{t^*}(x)$. We calculate
\begin{eqnarray*}
\tilde\rho([s,x]) & = & [s, \rho(x)],\\
\tilde\rho([t,\alpha_{t^*}(x)]) & = & [t, \rho(\alpha_{t^*}(x))],\\
&=& [t, \theta_{t^*}(\rho(x))],
\end{eqnarray*}
and so $(\tilde\rho([s,x]), \tilde\rho([t,y]))\in \g(\theta)^{(2)}$. Furthermore,
\begin{eqnarray*}
\tilde\rho([s,x])\tilde\rho([t,\alpha_{t^*}(x)]) & = & [s, \rho(x)][t, \theta_{t^*}(\rho(x))]\\
& = & [st, \rho(\alpha_{t^*}(x))],\\
&=& \tilde\rho([st,\alpha_{t^*}(x)]),\\
&=& \tilde\rho([s,x][t,\alpha_{t^*}(x)]),
\end{eqnarray*}
whence $\tilde\rho$ is a groupoid homomorphism.

Now we show that $\tilde\rho: \g(\alpha)_x\to \g(\theta)_{\rho(x)}$ is a bijection for all $x\in X$. If $[s, \rho(x)]\in \g(\theta)_{\rho(x)}$, then $\tilde\rho([s, x]) = [s, \rho(x)]$, and so $\tilde\rho$ is surjective. Now, take $[s, x], [t, x]\in \g(\alpha)_x$, and suppose that $[s, \rho(x)]= [t, \rho(x)]$. This implies that $[s]^\theta_{\rho(x)} = [t]^\theta_{\rho(x)}$, which by Lemma \ref{rhofacts}.3 is equivalent to $[s]^\alpha_x = [t]^\alpha_x$, which gives us that $[s, x] = [t, x]$. Hence, $\tilde\rho:\g(\alpha)_x\to \g(\theta)_{\rho(x)}$ is bijective.
\end{proof}
\begin{lem}\label{borellemma}
 The maps $\rho: X\to \Ef$  and $\tilde\rho: \g(\alpha) \to \g(\theta)$ are Borel maps.
\end{lem}
\begin{proof}
Sets of the form $D^\theta_e$ together with their complements form a subbasis for the topology on $\Ef$. Since the Borel sets form a $\sigma$-algebra, so we need only check that for all $e\in E$ the set $\rho^{-1}(D^\theta_e)$ is Borel. A short calculation shows that $\rho^{-1}(D^\theta_e) = D_e^\alpha$.

Now, suppose $s\in \s$ and we have $e\in E$ such that $D_e^\theta\subset D_{s^*s}^\theta$. Then 
\[
\tilde\rho^{-1}(\Theta(s, D_e^\theta)) =\Theta(s,D_e^\alpha)
\]
which is an open set. Furthermore, 
\[
\tilde\rho^{-1}(\Theta(s, D_{s^*s}^\theta\setminus D_e^\theta)) =\Theta(s,D_{s^*s}^\alpha)\setminus\Theta(s,D_e^\alpha)
\]
which is a Borel set. Sets of these types generate the topology of $\g(\theta)$, so $\tilde\rho$ is Borel.
\end{proof}
\begin{proof}[Proof of Theorem \ref{maintheorem}]
We need only prove that (a)$\Rightarrow$(b), because (b)$\Rightarrow$(a) is obvious. That (a)$\Rightarrow$(b) follows from Proposition \ref{dbijectiveamenable}, Lemma \ref{rhohomo}, and Lemma \ref{borellemma}.
\end{proof}

%\begin{cor}
%Let $\s$ be an inverse semigroup. Then the following are equivalent.
%\begin{enumerate}\addtolength{\itemsep}{-0.5\baselineskip}
%\item[(a)] The canonical action $\theta: \s\curvearrowright \Ef$ is amenable.
%\item[(b)] Every action of $\s$ is amenable.
%\item[(c)] $\s$
%\end{enumerate}
%\end{cor}
We close with two remarks regarding our result.

\begin{rmk}\label{weakcontainment}
To an \'etale groupoid $\g$ one can associate C*-algebras $C^*(\g)$ and $C^*_r(\g)$, called the {\em C*-algebra of $\g$} and the {\em reduced C*-algebra of $\g$} respectively. In this work we are not concerned with the specifics (and the interested reader is directed to \cite{R80} for more details), but there is always a surjective $*$-homomorphism $\lambda: C^*(\g)\to C^*_r(\g)$. If $\g$ is amenable, then $\lambda$ is an isomorphism, \cite[Proposition 6.1.8]{AR00}. 

{\em Weak containment} for an inverse semigroup was defined in \cite{DP85}, and in \cite{Mi10} it was argued that it is a good candidate for the definition of amenability for inverse semigroups. It is true that an inverse semigroup $\s$ with universal action $\theta$ satisfies weak containment if and only if the map $\lambda: C^*(\g(\theta))\to C^*_r(\g(\theta))$ is an isomorphism, see \cite[Theorem 4.4.2]{Pa99}. Hence, the equivalent conditions of Theorem \ref{maintheorem} imply weak containment, though it is not known whether the converse holds, see \cite[Remark 6.1.9]{AR00} and \cite{Wi15}.
\end{rmk}
\begin{rmk}\label{tightcounterexample}
Another intrinsic action of an inverse semigroup $\s$ is that on a subspace of $\Ef$, called the {\em tight spectrum} of $\s$. One considers in $\Ef$ the subset of all {\em ultrafilters}, that is, the filters which are not properly contained in another filter. The tight spectrum is then the closure in $\Ef$ of the set of all ultrafilters, and is denoted $\Et$. The action of $\s$ on $\Ef$ restricts to an action on $\Et$, and the resulting groupoid of germs is called the {\em tight groupoid} of $\s$ and is denoted $\gt(\s)$. For details of this construction, the reader is directed to \cite{Ex08}.

Our thought when setting out to investigate amenability of inverse semigroup actions was that perhaps the following entry could be added to Theorem \ref{maintheorem}:
\begin{enumerate}\addtolength{\itemsep}{-0.5\baselineskip}
\item[(c)]The canonical action $\theta: \s\curvearrowright \Et$ is amenable.
\end{enumerate}
This however is not true, as evidenced by the following counterexample which was relayed to us by Benjamin Steinberg.

Let $G$ be some discrete nonamenable group (such as the free group on two elements), and let $G^0$ denote the inverse
semigroup obtained by adjoining an ad-hoc zero element $0$ to $G$. Now, let $\s$ be the inverse semigroup obtained by
adjoining a further ad-hoc zero element $0'$ to $G^0$. The set of idempotents for this inverse semigroup is $E = \{1_G, 0, 0'\}$, and
\[
\Ef = \{\{1_G, 0\}, \{1_G\}\}
\]
\[
\Et = \{\{1_G, 0\}\}.
\]
Let $\xi = \{1_G, 0\}$ so that $\Et = \{\xi\}$. Suppose we have two germs $[s, \xi], [t, \xi]$ in $\gt(\s)$. We note that neither $s$ nor $t$ can be equal to $0'$. Since $0\in\xi$, $\xi\in D^\theta_0$, and $s0 = t0 = 0$, we have $[s, \xi] =[t, \xi]$ and so $\gt(\s)$ is the trivial (one-point) groupoid, hence amenable. However, $\g(\theta)$ is the union of the nonamenable group $G$ with a single point, and so is not amenable.
\end{rmk}

{\bf Acknowledgment:} We are thankful to Benjamin Steinberg for relaying to us the example in Remark \ref{tightcounterexample}, and to the referee for a careful reading.
\bibliography{C:/Users/Charles/Dropbox/Research/bibtex}{}
\bibliographystyle{alpha}
\end{document}